\documentclass[12pt, leqno]{article}

\usepackage{amsmath,amsthm}
\usepackage{amssymb}

\usepackage{enumerate}

\usepackage{graphicx}
\usepackage{xypic}

\newtheorem{thm}{Theorem}[section]
\newtheorem{theorem}{Theorem}[section]

\newtheorem{lem}[thm]{Lemma}
\newtheorem{prop}[thm]{Proposition}

\theoremstyle{definition}

\numberwithin{equation}{section}

\usepackage{amsfonts}
\usepackage{mathrsfs}
\usepackage{enumerate}
\usepackage{longtable}
\usepackage[citecolor=red]{hyperref}

\usepackage{url}
\usepackage{multirow}

\begin{document}
	
	\title{Computing modular Galois representations for small $\ell$}
	
	\author{Peng Tian\\
		Department of Mathematics\\
		East China University of Science and Technology\\
		200237, Shanghai, P. R. China}
	
	%\date{}
	
	\maketitle
	
	%% Classification and key words; note that the 2010 classification is used:
	
	%\renewcommand{\thefootnote}{}

	%\footnote{2010 \emph{Mathematics Subject Classification}: 11-04, 11Fxx, 11G30, 11Y40.}
	
	%\footnote{\emph{Key words and phrases}: modular Galois representations, modular forms, modular curves, Jacobian, %Ramanujan's tau function, polynomials.}

	\renewcommand{\thefootnote}{\arabic{footnote}}
	\setcounter{footnote}{0}

 {\begin{center}
\parbox{14.5cm}{\begin{abstract}
 In this paper we describe an algorithm for computing mod $\ell$ Galois representations associated to modular forms of weight $k$ when $\ell <k-1$. As applications, we use this algorithm to explicitly compute the cases with $\Delta_{k}$ for $k=16,20, 22, 26$ and all the unexceptional primes $\ell$ with $\ell <k-1$.\vspace{-3mm}
\end{abstract}}\end{center}}

%  Keyword is required.
% \keywords{modular Galois representations, mod $\ell$ modular forms, twist of Galois representations}

%  \subjclass is required.
% \MSC{11F80, 11F33, 11G30, 11G18}

%%%%%%%%%%%%%%%%%%%%%%%%%%%%%%%%%%%%%%%%%%%%%%%%%%%%%%%%%%%%
\renewcommand{\baselinestretch}{1.2}

%%%%%%%%%%%%%%%%%%%%%%%%%%%%%%%%%%%%%%%%%%%%%%%%%%%%%%%%%%%%
%% Text of article.
%%%%%%%%%%%%%%%%%%%%%%%%%%%%%%%%%%%%%%%%%%%%%%%%%%%%%%%%%%%%
%    Section headings
%\baselineskip 11pt\parindent=10.8pt  \wuhao

\section{Introduction}

	In the book \cite{book}, S. J. Edixhoven, J.-M. Couveignes, et al. proposed a polynomial time algorithm to compute the mod $\ell$ Galois representations $\rho_{f,\ell}$ associated to level one modular forms $f \in S_{k}(SL_2(\mathbb{Z}))$. In fact  $\rho_{f,\ell}$ can be described by a certain polynomial $P_{f,\ell}\in \mathbb{Q}[x]$ of degree $\ell^{2}-1$ whose splitting field is the fixed field of the kernel ker($\rho_{f,\ell}$). This algorithm has been generalized to forms of arbitrary levels by Bruin \cite{bruin}. Moreover the associated projective representation   $\tilde{\rho}_{ f,\ell}$ can be described as the splitting field of a suitable polynomial $\tilde P_{f,\ell}\in \mathbb{Q}[x]$ of degree $\ell+1$. 
	
	Let $\Delta_{k}$ be the unique cusp form of level $1$ and weight $k$ with $k=12,16,18,20,22,26$. In practice, this algorithm has been first implemented by J. Bosman \cite[Chapter 7]{book} to evaluate  $\tilde P_{\Delta_{k},\ell}$ for $\ell\le 23$ and $k\le \ell+1$. Recently in \cite{mascot}  and \cite{peng}, this algorithm has been improved and more polynomials $\tilde P_{\Delta_{k},\ell}$ have been explicitly computed when $\ell\le43$.
	
	In the book \cite{book}, the authors dealed with the case with $\ell<k-1$ by twisting the representations and then boil down to the cases with $k\le\ell+1$. In fact, for a form of level one and weight $k$ with $\ell<k-1$, in \cite[Proposition 2.5.18]{book}  they showed a method to obtain a form of weight $k'\le \ell+1$ such that the two  Galois representations associated the two forms are isomorphic. In practice, however, no one has implemented the algorithm to calculate the polynomials  for the cases with $\ell<k-1$. 
	
	In this paper, we shall discuss the algorithm for computing mod $\ell$ Galois representations associated to modular forms of weight $k$ when $\ell <k-1$.
	
	In Section \ref{sec:modlforms}, we first show  a generalization of Sturm bound theorem \cite[Theorem 2]{sturm} to mod $\ell$ modular forms, which gives an explicit method to identify two forms by the coefficients of the $q$-expansions.  Then in Section \ref{sec:algorithm},  we use this result to present a method, for a form of type  $(N,k,\varepsilon )$ with $\ell<k-1$, to obtain a form of $(N,k',\varepsilon )$ with $k'\le \ell+1$ such that the two  Galois representations associated the two forms are isomorphic. In fact this is a generalization of \cite[Proposition 2.5.18]{book} to not only level one forms but ones with arbitrary levels. Consequently, it suffices for us to do explicit calculations only for the cases with $k\le \ell+1$.  Finally, for purpose of practical calculations,  in Section \ref{sec:algorithm}, we prove the corresponding results for the projective representations and then present the algorithm for the projective case.
	
	In the last section, we apply the algorithm in Section 3 to do explicit computations for $\Delta_{k}$ with $k=16,20, 22, 26$ and all the unexceptional primes $\ell$ with $\ell <k-1$. We first show the $\theta$-twist forms in Table \ref{table:twistforms}  and then  obtain the projective polynomials $\tilde P_{\Delta_{k},\ell}(x)$ associated to the mod $\ell$ projective Galois representation  $\tilde{\rho}_{ \Delta_{k}}$ which are shown in the Table \ref{table:polynomials}. 
	
	All the explicit  computations of this paper have been done in the open source software SAGE \cite{sage}.

\section{Mod $\ell$ modular forms } \label{sec:modlforms}
Throughout this paper, we suppose $\ell\ge5$ to be a prime and denote $\overline{\mathbb{F}}_{\ell}$  the algebraic closure of $\mathbb{F}_{\ell}$.

\subsection{Modular forms of type $(N,k,\varepsilon)$}  \label{subsec:definition}

The mod $\ell$ modular forms are first developed by J-P. Serre \cite{serresd1} and  H. P. F. Swinnerton-Dyer \cite{sd1}, and generalized by N. M. Katz \cite{katzmodlforms}. In this subsection we give a brief review of the theory of mod $\ell$ modular forms. For the details, we refer to \cite{gross} and \cite[Section 2]{ediweight}.

Let $\ell$ be a prime and $N\ge1$ be prime to $\ell$.  The congruence subgroup $\varGamma _{1}(N)$ of level $N$ is
\begin{center}
$\varGamma _{1}(N) = \left\lbrace 
\begin{pmatrix}
a & b\\
c & d
\end{pmatrix}  \in SL(2,\mathbb{Z}) \ \mid \ c \equiv 0 \ mod \ N , \  \ \ a\equiv d\equiv 1 \ mod \ N
\right\rbrace.$
\end{center}

Let $X_{1}(\ell)$ be the modular curve associated to $\varGamma_{1}(\ell)$. Let $k>0$ be an even integer. Let $E$ be a generalized elliptic curve over a scheme $S$ and $\alpha:(\mathbb{Z}/N\mathbb{Z})_{S}\hookrightarrow E$ be an embdedding of group schemes. Then we have an invertible sheaf $\omega_{E/S}$ of relative differentials $\varOmega^{1}_{E/S}$ and zero section $0$:

$$\omega_{E/S}:= 0^{*} \varOmega^{1}_{E/S}.$$
Then a modular form $f$ of type $(N,k)$ over $\overline{\mathbb{F}}_{\ell}$ is a law, compatible with   Cartesian squares, that assigns to each pair $(E,\alpha)$ a section of $\omega_{E/S}^{\otimes k}$.

If $N>4$, then  modular forms $f$ of type $(N,k)$ are the global sections of the invertible sheaf $\underline{\omega}^{\otimes k}$ on the modular curve $X_{1}(N)_{\overline{\mathbb{F}}_{\ell}}$. If $N\le4$, then modular forms $f$ of type $(N,k)$ are the sections of $\underline{\omega}^{\otimes k}$ over the complement of some points in the curve $X_{1}(N)_{\overline{\mathbb{F}}_{\ell}}$.

The $q$-expansions of mod $\ell$ modular forms $f$ at cusp $\infty$ of $\varGamma_{1}(N)$ have been given by evaluating $f$ on $(E_{q},\alpha)$, where $q=e^{2\pi iz}$ and $E_{q}$ is the Tate curve over $\overline{\mathbb{F}}_{\ell}[[q]](q^{-1})$. More precisely, the $q$-expansions of $f$ at $\infty$ is the the power series $f(E_{q},\alpha)/(dt/t)^{\otimes k}\in\overline{\mathbb{F}}_{\ell}[[q]]$, where $dt/t$ is the standard differential on $E_{q}$. This in fact coincides with the usual $q$-expansions of modular forms, since $(E_{q},\alpha)$ corresponds to a neighbourhood of the cusp $\infty$ in the completed up half plane $\mathcal{H}^{*}=\mathcal{H}\cup \mathbb{Q} \cup \infty$, where $\mathcal{H}$ is the up half plane. As usual, we denote the $n$-th coefficient of the $q$-expansion by $a_{n}(f)$.

 Let $\varepsilon:(\mathbb{Z}/(N\ell)\mathbb{Z})^{\ast}\rightarrow \overline{\mathbb{F}}_{\ell}$ be a Dirichlet character. Define an action of $\mathbb{Z}/(N\ell)\mathbb{Z})^{\ast}$ on mod $\ell$ form $f$ by
 $$(\langle a\rangle^{*})(E/S,\alpha)=f(E/S,a\alpha),  \ \ \ a\in \mathbb{Z}/(N\ell)\mathbb{Z})^{\ast}.$$
 A modular form $f$ of type $(N,k)$ is called  a form of type $(N,k,\varepsilon)$ if it satisfies 
 $$(\langle a\rangle^{*})(E/S,\alpha)=\varepsilon(a)f.$$

One can also define Hecke operators $T_{p}$ that are coincide with the usual Hecke operators. For instance, we have that all the $T_{p}$ commute with each other and the eigenvalues determine the $q$-expansions of $f$ up to a constant factor.

 A modular form $f$ is  called cusp form if $a_{0}(f)=0$. A modular form $f$ of type $(N,k,\varepsilon)$ is said to be an eigenform if it is an eigenvector for all the Hecke operators $T_{p}$ with $p \nmid N\ell$. An eigenform $f$ is said to be normalized if $a_{1}(f)=1$.

\subsection{Operator $\theta$ and Hasse invariant $A$}
Let $\theta=q\frac{d}{dq}$  be the classical differential operator. If $f$ is an eigenform of type $(N,k,\varepsilon)$, in \cite[Section 2.1]{katz}, it is shown that $\theta f$ is an eigenform of type $(N,k+\ell+1,\varepsilon)$.

Let $A$ be the Hasse invariant of the Tate curve $E_{q}$ over $\overline{\mathbb{F}}_{\ell}[[q]](q^{-1})$, then we have
\begin{lem}\label{hasseinvariant}
\quad The Hasse invariant $A$ is given by $A=(dt/t)^{\otimes \ell-1}$. Hence $A$ is a mod $\ell$ modular form of type $(1,\ell-1,1)$
\end{lem}
\begin{proof}\quad  This is the Proposition 1.9 $c)$ in \cite{gross}.
\end{proof}
From this lemma, we know the $q$-expansion of $A$ is $1$. For two forms of types $(N,k_{1},\varepsilon)$ and $(N,k_{2},\varepsilon)$, respectively, with $k_{1}\equiv k_{2} \mod \ell-1$, then we can view the two forms as of the some type by multiplying one form by suitable powers of $A$. This can be used to prove the following lemma which is a generalization of Sturm bound theorem to modular forms of different weights.

\begin{prop}  \label{differentweights}\quad
	Let $f_{1}$ and $f_{2}$ be two normalized eigenforms of type $(N,k_{1},\varepsilon)$ and $(N,k_{2},\varepsilon)$, respectively. Let $k=max\{k_{1},k_{2}\}$. Suppose that $k_{1}\equiv k_{2} \mod \ell-1$ and  $a_{m}(f_{1}) = a_{m}(f_{2})$ in $\overline{\mathbb{F}}_{\ell}$ for all $m$ with  $m \le \frac{k[SL_{2}(\mathbb{Z}):\varGamma_{1}(N)]}{12}$. Then  $f_{1} = f_{2}$.
\end{prop}
\begin{proof}\quad
	Let $A$ be the Hasse invariant. Without loss of generality, we suppose $k_{1}\le k_{2}$. Then by Lemma \ref{hasseinvariant}, the form $A^{(k_{2}-k_{1})/(\ell-1)} f_{1}$ is an eigenform of type $(N,k_{2},\varepsilon)$. we know $A=1$, and this implies that the form $f_{1}$ is also a form of type $(N,k_{2},\varepsilon)$. Since we have  $a_{m}(f_{1}) = a_{m}(f_{2})$ in $\overline{\mathbb{F}}_{\ell}$ for all $m$ with  $m \le \frac{k[SL_{2}(\mathbb{Z}):\varGamma_{1}(N)]}{12}$, it follows from Sturm's theorem that   $f_{1} = f_{2}$.
\end{proof}

The following well-known theorem takes an important role for our computations.

\begin{theorem} \label{twist} \quad
Let $f$ be a normalized eigenform of type $(N,k,\varepsilon)$, then there exist $i$ and $k'$ with $0\le i\le \ell-1, \ k'\le \ell+1$, and a normalized eigenform of type $(N,k',\varepsilon)$, such that $f=\theta^{i}g$.
\end{theorem}
\begin{proof}\quad
See \cite[Theorem 3.4]{ediweight}.
\end{proof}

\section{Computing mod $\ell$ Galois representations for small $\ell$}\label{sec:algorithm}
In this section, we shall present the algorithm for computing mod $\ell$ Galois representations associated to modular forms of weight $k$ when $\ell <k-1$. 

\subsection{Galois representations and twists}

 P. Deligne \cite{deligne} proves  the following well known theorem:

\begin{theorem}[Deligne] \label{deligne} \quad
Let $f$ be an eigenform of type $(N,k,\varepsilon)$. Then there exists a continuous semi-simple representation
 \begin{equation} \label{galoisrep}
  \rho_{f}: Gal(\overline{\mathbb{Q}}|\mathbb{Q}) \rightarrow GL_{2}(\overline{\mathbb{F}}_{\ell}),
 \end{equation}
 that is unramified outside $N\ell$, and for all primes $p\nmid N\ell$ the characteristic polynomial of  $\rho_{f} (Frob_{p}) $ satisfies  in  $\overline{\mathbb{F}}_{\ell}$
\begin{equation} \label{charpol}
charpol( \rho_{f} (Frob_{p})) =  x^2-a_{p}(f)x+ \varepsilon(p)p^{k-1}.
\end{equation}
Moreover, $\rho_{f}$ is unique up to isomorphism.
\end{theorem}

Let $f = \sum_{n>0} a_{n} (f) q^{n}$ be an eigenform. Then the eigenform $\theta f$ has $q$-expansion $\sum_{n>0} na_{n} (f) q^{n}$. It follows from the above theorem  that
$$  \rho_{\theta f}=  \rho_{f} \otimes \chi_{\ell}, $$
where $\chi$ is the mod $\ell$ cyclotomic character.

For an eigenform $f$ of type $(N,k,\varepsilon)$ with $\ell<k-1$, it follows from Theorem \ref{twist} that there exist an integer $i$ and an egienform $g$ of type $(N,k',\varepsilon)$ with $k' \le \ell+1$, such that 
$$\rho_{ f} \cong \rho_{g} \otimes \chi^{i}_{\ell}.$$ 
Moreover, we have the following theorem to determine such $i$ and $k'$,  which is a generalization of \cite[Proposition 2.5.18]{book} to not only level one forms but ones with arbitrary levels. 

\begin{theorem} \label{kandi} \quad
	Let $f_{1}$ and $f_{2}$ be two normalized eigenforms of type $(N,k_{1},\varepsilon)$ and $(N,k_{2},\varepsilon)$, respectively.  Then  $\rho_{ f_1}$ and  $\rho_{f_2} \otimes \chi^{i}_{\ell}$ are isomorphic if and only if $k_{1}\equiv k_{2}+2i \mod \ell-1$ and  $a_{p}(f_{1}) = p^{i}a_{p}(f_{2})$ in $\overline{\mathbb{F}}_{\ell}$ for all primes $p$ with $p\ne N\ell$ and  $p \le \frac{\ell(\ell+1)[SL_{2}(\mathbb{Z}):\varGamma_{1}(N)]}{12}$.
\end{theorem}
\begin{proof}\quad
	We first assume that $\rho_{ f_1}$ and  $\rho_{f_2} \otimes \chi^{i}_{\ell}$ are isomorphic. Then by  (\ref{charpol}), we have 
	 $$\varepsilon \chi_{\ell}^{k_{1}-1}= \varepsilon \chi_{\ell}^{k_{2}-1+2i}.$$
	 Hence we have  $k_{1}\equiv k_{2}+2i \mod \ell-1$. Moreover we have 
	 $$tr(\rho_{ f_1}(\mathrm{Frob}_{p}))=tr((\rho_{f_2} \otimes \chi^{i}_{\ell})(\mathrm{Frob}_{p}))$$ 
	 for any $p\ne \ell$, and hence we have  in $\overline{\mathbb{F}}_{\ell}$ 
	 $$a_{p}(f_{1}) = p^{i}a_{p}(f_{2})$$
	  for all primes $p$ with $p\ne N\ell$ and  $p \le \frac{\ell(\ell+1)[SL_{2}(\mathbb{Z}):\varGamma_{1}(N)]}{12}$.
	
	For the other direction, we assume that $k_{1}\equiv k_{2}+2i \mod \ell-1$ and  
	$$a_{p}(f_{1}) = p^{i}a_{p}(f_{2})$$
	 in $\overline{\mathbb{F}}_{\ell}$ for all primes $p$ with $p\ne N\ell$ and  $p \le \frac{\ell(\ell+1)[SL_{2}(\mathbb{Z}):\varGamma_{1}(N)]}{12}$.
	 
	It follows from Theorem \ref{twist} that there exist an integer $j$ with $0\le j \le \ell-1$ and a normalized eigenform $g_{1}$ of type  $(N,k_{g_{1}},\varepsilon)$ with  $k_{g_{1}}\le\ell+1$ such that $f_{1} = \theta^{j}g_{1}$ in $\overline{\mathbb{F}}_{\ell}$.
	This implies that we have a form $f'_{1}=\theta^{j}g_{1}$ of  type  $(N,k_{1}',\varepsilon)$ with $k_{1}'\le \ell(\ell+1)$ such that  $\rho_{ f_{1}}$ and  $\rho_{f_{1}'}$ are isomorphic. By (\ref{charpol}) we have 	\begin{equation}\label{kmod1}
     k_{1}'\equiv k_{1}\mod \ell-1.
	\end{equation}
	
	For the same reason, we have a form $f'_{2}$ of  type  $(N,k_{2}',\varepsilon)$ with $k_{2}'\le \ell(\ell+1)$ such that     $ a_{p}(f_{2}') = p^{i}a_{p}(f_{2})$ and  $\rho_{f_2'}$ is isomorphic to $\rho_{f_2} \otimes \chi^{i}_{\ell}$.  By the argument in the first paragraph of the proof, we have   
	\begin{equation}\label{kmod2}
	k_{2}'\equiv k_{2}+2i \mod \ell-1.
	\end{equation}
	
	Then by (\ref{kmod1}), (\ref{kmod2}) and the assumption, we have 
	$$k_{1}'\equiv k_{1} \equiv k_{2}+2i \equiv k_{2}' \mod \ell-1$$ and   in $\overline{\mathbb{F}}_{\ell}$
	$$a_{p}(f_{1}') = a_{p}(f_{1}) = p^{i}a_{p}(f_{2}) = a_{p}(f_{2}')$$
	 for all primes $p$ with $p\ne N\ell$ and  $p \le \frac{\ell(\ell+1)[SL_{2}(\mathbb{Z}):\varGamma_{1}(N)]}{12}$. 
	 
	 Moreover, we know that $a_{\ell}(f_{1}')$ and  $a_{\ell}(f_{2}')$ are congruent to $0$ modulo $\ell$ since they are divided by a positive power of $\ell$. Therefore, we have  $a_{m}(f_{1}') = a_{m}(f_{2}')$  in $\overline{\mathbb{F}}_{\ell}$ for all  $m$ with  $m \le \frac{\ell(\ell+1)[SL_{2}(\mathbb{Z}):\varGamma_{1}(N)]}{12}$. By Proposition \ref{differentweights}, we then have that  $f_{1}' = f_{2}'$ and therefore $\rho_{ f_{1}'}$ and  $\rho_{f_{2}'}$ are isomorphic. Hence  $\rho_{ f_1}$ and  $\rho_{f_2} \otimes \chi^{i}_{\ell}$ are isomorphic. This completes the proof.
\end{proof}

In \cite[Theorem 3.5]{bruin}, the author gives a proof of  a more elaborate  result, since the purpose of the thesis is to theoretically prove that the algorithm is in polynomial time. In our paper, we intend to do explicit computations and therefore the result in Theorem \ref{kandi} can meet our purpose.

\subsection{The Algorithm}
In this subsection, we shall describe the algorithm for computing the mod $\ell$ Galois representations associated to modular forms. In fact, we have the following result which is first proposed by S. Edixhoven and J.-M. Couveignes \cite{book} for modular forms of level one and then generalized  to forms of arbitrary levels by Bruin \cite{bruin}.

\textit{There exists a deterministic  algorithm that computes the $\mathrm{mod} \ \ell$ Galois representation associated to  modular forms in time polynomial in $\ell$.}

Let $f$ be a cuspidal  normalized eigenforms of type $(N,k,\varepsilon)$. If  $\ell<k-1$, Theorem \ref{twist} and \ref{kandi} allow us to explicitly obtain normalized eigenforms $f'$ of type $(N,k',\varepsilon)$ with $2\le k' \le \ell+1$ such that  $\rho_{f}$ and  $\rho_{f'} \otimes \chi^{i}_{\ell}$ are isomorphic. Thus the question boils down to the case of $2 \le k \leq \ell+1$.

In \cite[Theorem 2.2]{report}, the author shows that if $2 < k \leq \ell+1$ and $\rho_{f,\lambda}$ is ireducible, then there is a  cuspidal  normalized eigenforms $f_{2}$ of type $(N\ell,2,\varepsilon_2)$ such that $\rho_{f}$ is isomorphic to $\rho_{f_{2}}$. Therefore, for any $p \nmid N\ell$, this reduces the questions to the cases of $k=2$.

Now suppose that $\rho_{f}$ is a mod $\ell$ Galois representation associated to a  cuspidal  normalized eigenforms of type $(N,2,\varepsilon)$. Let $X_{1}(\ell)$ be the modular curve associated to $\varGamma_{1}(\ell)$ and let $J_{1}(\ell)$ denote its Jacobian. Denote $\mathbb T$ the subring of End$(J_{1}(\ell))$ generated by the Hecke operators $T_p$ over $\mathbb Z$. Then
$$\mathbb T=\mathbb{Z}[  T_{n}, \langle n\rangle : n \in \mathbb Z_{+} \ \mathrm{and} \ (n,\ell)=1].$$

Now define a ring homomorphism
$$\theta: \mathbb{T} \rightarrow \mathbb F_{\lambda} ,$$
given by 
$$\langle d \rangle \mapsto \varepsilon(d) \ \  and  \ \  T_{n} \mapsto a_{n}(f).$$

Let $\mathfrak m$ denote the maximal ideal ker$\theta$ and then  $\mathbb{T}/\mathfrak{m}\subset \overline{\mathbb{F}}_{\ell}$. Moreover we let
\begin{displaymath} 
V=J_{1}(\ell)(\overline{\mathbb{Q}})[\mathfrak{m}]=\{x \in J_{1}(\ell)(\overline{\mathbb{Q}}) \ | \ tx=0 \ \mathrm{for}  \ \mathrm{all} \ t \ \mathrm{in} \ \mathfrak m \}.
\end{displaymath} 
 Then we have
\begin{theorem} \label{v}  \quad
	The set $V$ is  a 2-dimensional $\mathbb{T}/\mathfrak{m}$-linear subspace of $J_{1}(\ell)(\overline{\mathbb{Q}})[\ell]$. Moreover,  the representation
	$$
	\rho: Gal (\mathbb{\overline{Q}}/\mathbb{Q}) \rightarrow \mathrm{Aut}(V)
	$$
	is isomorphic to the modular Galois representation $\rho_{f}$.
\end{theorem}
\begin{proof}\quad
	See \cite[Section 3.2 and 3.3]{ribetstein}).
\end{proof}

	Let $L$ be the fixed field of the kernel ${\rm ker}(\rho_{f})$ of the Galois representation $\rho_{f}$. Then the representation $\rho_{f}$ can factor through as:
\begin{displaymath} 
\xymatrix{Gal (\mathbb{\overline{Q}}|\mathbb{Q}) \ar[rr]^{\rho_{f}}\ar[dr]_{\pi}  &   & GL_{2}(\overline{\mathbb{F}}_{\ell})  \\
	&  Gal ( L|\mathbb{Q}) \ar[ur]^{\phi} &
}
\end{displaymath}
where $\pi$ is the canonical restriction map and $\phi$ is the isomorphism between  Gal$( L|\mathbb{Q})$ and  the image ${\rm im}(\rho_{f})$ of $\rho_{f}$. To compute $\rho_{f}$, it suffices to compute a suitable polynomial  $P_{f}\in \mathbb Q[x]$ of degree $\ell^2-1$ with 
\begin{displaymath}
P_{f}(x) = \prod_{P\in V-\{0\}}(x-h(P))
\end{displaymath}
for some suitable function $h$ in the function field of $X_1(\ell)$. Here $h(P)=\sum_{i=1}^g h(P_i)$ where  $g$ is the genus of $X_1(\ell)$, and $P_i$ are the points on $X_1(\ell)$ such that each divisor $P\in V-\{0\}$ can be written as $\sum_{i=1}^g (P_i) -gO$.   In fact, it can be shown that fixed field of $\rho_{f}$ is actually  the splitting field of  $P_{f}\in \mathbb Q[x]$.

\subsection{Projective Galois representations} \label{subsec:algorithmproj}

Composed with the canonical projection map $GL_{2}(\mathbb{F}_{\lambda})\rightarrow PGL_{2}(\mathbb{F}_{\lambda})$, the representation $\rho_{f,\lambda}$ in (\ref{galoisrep}) gives a projective representation 
$$\tilde{\rho}_{f}: Gal(\overline{\mathbb{Q}}|\mathbb{Q}) \rightarrow PGL_{2}(\overline{\mathbb{F}}_{\ell}).$$

Now we apply Theorem \ref{kandi} to the case of projective representation and then we have

\begin{theorem} \label{kandiofproj} \quad
	Let $f_{1}$ and $f_{2}$ be two normalized eigenforms of type $(N,k_{1},\varepsilon)$ and $(N,k_{2},\varepsilon)$, respectively. Suppose that $k_{1}\equiv k_{2}+2i \mod \ell-1$ and  $a_{p}(f_{1})  = p^{i}a_{p}(f_{2})$ in $\overline{\mathbb{F}}_{\ell}$ for all primes $p$ with $p\ne N\ell$ and  $p \le \frac{\ell(\ell+1)[SL_{2}(\mathbb{Z}):\varGamma_{1}(N)]}{12}$. Then  $\tilde{\rho}_{ f_1}$ and  $\tilde{\rho}_{f_2} $ are isomorphic.
\end{theorem}

\begin{proof}\quad
It follows from Theorem \ref{kandi} that $\tilde{\rho}_{ f_1}$ and  $\tilde{\rho}_{f_2} \otimes \chi^{i}_{\ell}$ are isomorphic. For any $\sigma \in$  Gal$(\overline{\mathbb{Q}}|\mathbb{Q})$, we have 
$$\rho_{f_2} \otimes \chi^{i}_{\ell}(\sigma)=\rho_{f_2}(\sigma) \cdot\chi^{i}_{\ell}(\sigma).$$
 In $PGL_{2}(\mathbb{F}_{\lambda})$, we have 
 $$\overline{\rho_{f_2}(\sigma)}=\overline{\rho_{f_2}(\sigma) \cdot\chi^{i}_{\ell}(\sigma)},$$
  and hence  $\tilde{\rho}_{f_2} \otimes \chi^{i}_{\ell}=\tilde{\rho}_{f_2}$. Here, as usual, the bar denotes the quotient by the subgroup of $GL_{2}(\overline{\mathbb{F}}_{\ell})$ consisting of scalar matrices. This implies $\tilde{\rho}_{ f_1}$ and  $\tilde{\rho}_{f_2} $ are isomorphic.
\end{proof}

Now we can describe the algorithm for computing the projective Galois representation $\tilde{\rho}_{f}$ associated to  an normalized eigenform of type $(N,k,\varepsilon)$ with $\ell<k-1$.

First, by Theorem \ref{twist} and \ref{kandiofproj}, we can  explicitly obtain a normalized eigenform $f'$ of type $(N,k',\varepsilon)$ with $2\le k' \le \ell+1$ such that   $\tilde{\rho}_{ f}$ and  $\tilde{\rho}_{f'} $ are isomorphic. Thus our computations boil down to the case of $2 \le k \leq \ell+1$. Then again we can reduce the question to weight $2$ as the same arguments in the previous section. Finally we can compute a suitable polynomial instead for the following reason.

	Let  $K$ be the fixed field of ${\rm ker}(\tilde{\rho}_{f})$, the representation $\tilde{\rho}_{f}$ can factor through as:
\begin{displaymath} 
\xymatrix{
	& GL_{2}(\mathbb{F}_{\lambda}) \ar[dr] & \\
	Gal (\mathbb{\overline{Q}}|\mathbb{Q}) \ar[rr]^{\tilde{\rho}_{f,\lambda}}\ar[dr]_{\pi}  \ar[ur]^{\rho_{f,\lambda}}&   & PGL_{2}(\mathbb{F}_{\lambda})  \\
	&  Gal ( K|\mathbb{Q}) \ar[ur]^{\varphi} &
}
\end{displaymath}
where $\pi$ is the canonical restriction map and $\varphi$ is the isomorphism between  Gal$( L|\mathbb{Q})$ and  ${\rm im}(\tilde{\rho}_{f})$.
Let $V=J_{1}(\ell)(\overline{\mathbb{Q}})[\mathfrak{m}]$ be the 2-dimensional $\mathbb{T}/\mathfrak{m}$-linear subspace of $J_{1}(\ell)(\overline{\mathbb{Q}})[\ell]$ as in Theorem \ref{v}. Then the  projective line $\mathbb P(V)$ has $\ell+1$ points,  and it follows that the  fixed field of $\tilde{\rho}_{f}$ is in fact  the splitting field of a certain polynomial $\tilde P_{f}\in Q[x]$ of degree $\ell+1$ which is given by
\begin{equation}\label{projectivepolynomial}
\tilde P_{f}(x) = \prod_{A\subset \mathbb P(V)}(x-\sum_{P\in A-\{0\}}h(P)).
\end{equation}

\section{Examples}

	J. Bosman first did practical computations and obtained $\tilde P_{f}$ for modular forms $f$ of level 1 and of weight $k\le 22$, with $\ell \leq 23$. Recently, others improved the algorithm and computed the polynomials for more cases. See \cite{peng} and \cite{martin} for instance. As far as we know, all the polynomials  $\tilde P_{\Delta_{k},\ell}$ that have been computed in this method are shown in \cite[Section 7.5]{book} and \cite[Table 4]{martin}.  
	
	Note that all the computed polynomials are of the case with $k\le \ell+1$. In this section, we shall apply the algorithm described in Subsection \ref{subsec:algorithmproj} to compute some  polynomials associated to the Galois representations when $\ell<k-1$.

\subsection{Finding $\theta$-twist forms}
For $k=16,20,22$ and $26$, let $\Delta_k = \sum_{n>0}^{\infty}a_{n}q^{n}$ denote the unique cusp form of level $1$ and weight $k$. For a prime $\ell$, we let  $\tilde{\Delta}_k = \sum_{n>0}^{\infty}\tilde{a}_{n}q^{n}$, where $\tilde{a}_{n}$ means the reduction of $a_{n}$ mod $\ell$.  Then $\tilde{\Delta}_k$ is  a normalized cuspidal eigenform of type $(1,k,1)$. We denote by  $P_{\Delta_{k},\ell}$ the polynomial $P_{\Delta_k}(x)$ defined in (\ref{projectivepolynomial}) which describes the mod $\ell$ Galois representation associated to  $\tilde{\Delta}_k$.

 A prime $\ell$ is said to be exceptional if the image of $\rho_{f}$ does not contain $SL_{2}(\mathbb{F}_{\ell})$.
In the following table we list the unexceptional primes for $\Delta_k$  with $\ell<k-1$. Then for the  $(k,\ell)$  in the table, we shall compute the polynomials $\tilde P_{\Delta_{k},\ell}$.

\makeatletter\def\@captype{table}\makeatother

\begin{center}
	\begin{tabular}{|c|c|} 
		\hline
		$k$&$\ell$\\
		\hline
		$16$&13 \\
		\hline
		$20$&17\\
		\hline
		\multirow{2}{*}{22} & 11\\
		\cline{2-2}
		& 19 \\
		\hline
		\multirow{2}{*}{26} & 13 \\
		\cline{2-2}
		& 23 \\
		\hline
	\end{tabular}
       \caption{Small unexceptional primes for $\Delta_{k}$}
	\label{table:smallprimes}
	
\end{center}

For $\Delta_{k}$ and unexceptional prime $\ell$ with $(k,\ell)$ in Table \ref{table:smallprimes}, we apply Theorem \ref{kandiofproj} to find  normalized eigenforms $f$ of type  $(1,k',1)$ with $k'<\ell-1$ such that  $\tilde{\rho}_{\Delta_{k}}$ and  $\tilde{\rho}_{f} $ are isomorphic. We first find $k'$ and $i$ such that 
    $$k\equiv k'+2i \mod \ell-1.$$
then use SAGE \cite{sage} to verify $a_{p}(f_{1}) \equiv p^{i}a_{p}(f_{2}) \mod \ell$ for all primes $p$ with $p\ne N\ell$ and $p \le \frac{\ell(\ell+1)}{12}$.  In fact by Theorem \ref{twist}, such $k'$ and $i$ do exist and after doing some simple calculations we explicitly obtain the forms that satisfy the conditions in Theorem \ref{kandiofproj}. More precisely, it follows 
\begin{prop} \quad
	We  have the congruences
\begin{displaymath} \label{twistforms}
\Delta_{k} \equiv \theta^{i}\Delta_{k'} \mod \ell, 
\end{displaymath} 
where $k,\ell,i,k'$ are given by the following table.

\makeatletter\def\@captype{table}\makeatother

\begin{center}
		\begin{tabular}{|c|c|c|c|}
			\hline
			$k$&$\ell$& $i$&$k'$\\
				\hline
			$16$&$13$& $2$  &$12$ \\
				\hline
			$20$&$17$& $2$  &$16$ \\
				\hline
			\multirow{2}{*}{$22$} &$11$&$1$ & $12$\\
			\cline{2-4}
	         & $19$ & $2$ &$18$\\
           \hline
           \multirow{2}{*}{$26$} & $13$ &$1$ & $12$\\
           \cline{2-4}
           & $23$ & $2$ &$22$\\
           \hline
        \end{tabular}
       \caption{}
	\label{table:twistforms}
\end{center}
\end{prop}

\subsection{The polynomials}

It follows from Theorem \ref{kandiofproj} that $\tilde{\rho}_{ \Delta_{k}}$ and  $\tilde{\rho}_{\Delta_{k'}} $ are isomorphic and hence we have  
  $$\tilde P_{\Delta_{k},\ell}(x)= \tilde P_{\Delta_{k'},\ell}(x).$$ 
Fortunately, all the polynomials $\tilde P_{\Delta_{k'},\ell}(x)$ have been computed and shown in \cite[Section 7.5]{book}.
As a result, the corresponding polynomials $\tilde P_{\Delta_{k},\ell}(x)$ associated to the mod $\ell$ projective Galois representation  $\tilde{\rho}_{ \Delta_{k}}$ are shown in the following table.

\newpage

\renewcommand\arraystretch{1.8}

\begin{longtable}{|c|c|}
	
	\hline  
	($k$,$\ell$) &  $\tilde{P}_{\varDelta_{k},\ell}$   \\ [5pt] \hline
	(16, 13) &
	\begin{minipage}[t]{0.8\textwidth}
		$x^{14}+7*x^{13}+26*x^{12}+78*x^{11}+169*x^{10}+52*x^{9}-702*x^{8}-1248*x^{7}+494*x^{6}+2561*x^{5}+312*x^{4}-2223*x^{3}+169*x^{2}+506*x-215$
	\end{minipage}  \\  [15pt] \hline
	(20,17)&
	\begin{minipage}[t]{0.8\textwidth}
		$x^{18}-2*x^{17}-17*x^{15}+204*x^{14}-1904*x^{13}+3655*x^{12}+5950*x^{11}-3672*x^{10}-38794*x^{9}+19465*x^{8}+95982*x^{7}-280041*x^{6}-206074*x^{5}+455804*x^{4}+946288*x^{3}-1315239*x^{2}+606768*x-378241$
	\end{minipage} \\  [30pt] \hline
	(22,11)&
	\begin{minipage}[t]{0.8\textwidth}
		$x^{12}-4*x^{11}+55*x^{9}-165*x^{8}+264*x^{7}-341*x^{6}+330*x^{5}-165*x^{4}-55*x^{3}+99*x^{2}-41*x-111$
	\end{minipage} \\   [15pt] \hline
	(22,19)&
	\begin{minipage}[t]{0.8\textwidth}
		$x^{20}+10*x^{19}+57*x^{18}+228*x^{17}-361*x^{16}-3420*x^{15}+23446*x^{14}+88749*x^{13}-333526*x^{12}-1138233*x^{11}+1629212*x^{10}+13416014*x^{9}+7667184*x^{8}-208954438*x^{7}+95548948*x^{6}+593881632*x^{5}-1508120801*x^{4}-1823516526*x^{3}+2205335301*x^{2}+1251488657*x-8632629109$
	\end{minipage} \\  [40pt] \hline
	(26,13)&
\begin{minipage}[t]{0.8\textwidth}
	$x^{14}+7*x^{13}+26*x^{12}+78*x^{11}+169*x^{10}+52*x^{9}-702*x^{8}-1248*x^{7}+494*x^{6}+2561*x^{5}+312*x^{4}-2223*x^{3}+169*x^{2}+506*x-215$
\end{minipage} \\  [15pt] \hline	

(26,23)&
\begin{minipage}[t]{0.8\textwidth}
$x^{24}-11*x^{23}+46*x^{22}-1127*x^{20}+6555*x^{19}-7222*x^{18}-140737*x^{17}+1170700*x^{16}-2490371*x^{15}-16380692*x^{14}+99341324*x^{13}+109304533*x^{12}-2612466661*x^{11}+4265317961*x^{10}+48774919226*x^{9}-244688866763*x^{8}-88695572727*x^{7}+4199550444457*x^{6}-10606348053144*x^{5}-25203414653024*x^{4}+185843346182048*x^{3}-228822955123883*x^{2}-1021047515459130*x+2786655204876088$
\end{minipage} \\  [55pt] \hline
	
	\caption{Polynomials}
	\label{table:polynomials}
\end{longtable}

\section*{Acknowledgements}
This work was supported by National Natural Science Foundation of China (Grant No. 11601153) and  Fundamental Research Funds for the Central Universities (Grant No. 222201514319).


\end{document}